\documentclass[11pt, a4paper, reqno]{amsart}

\usepackage{amsmath, amssymb, amsthm}
\usepackage[T1]{fontenc}
\usepackage[utf8]{inputenc}
\usepackage[all]{xy}
\usepackage[
left=2.2cm,right=2.2cm,top=2.5cm,bottom=2.5cm,a4paper
  ]{geometry}
\usepackage{hyperref}
\usepackage{indentfirst}
\usepackage{algorithmic,algorithm}
\usepackage[last]{changes}
\usepackage{graphics,graphicx}
\definechangesauthor[color=blue]{KY}

\def\A{{\mathbb A}}

\def\C{{\mathbb C}}

\def\F{{\mathbb F}}

\def\P{{\mathbb P}}
\def\Q{{\mathbb Q}}
\def\R{{\mathbb R}}

\def\Z{{\mathbb Z}}

\theoremstyle{plain}
\newtheorem{theorem}{Theorem}[section]
\newtheorem{lemma}[theorem]{Lemma}
\newtheorem{proposition}[theorem]{Proposition}
\newtheorem{corollary}[theorem]{Corollary}
\newtheorem{definition}[theorem]{Definition}
\newtheorem{remark}{Remark}
\newtheorem{definition and lemma}[theorem]{Definition and Lemma}

\theoremstyle{definition} 

\newtheorem{example}[theorem]{Example}

\theoremstyle{remark} 

\DeclareFontFamily{U}{wncy}{}
\DeclareFontShape{U}{wncy}{m}{n}{<->wncyr10}{}
\DeclareSymbolFont{mcy}{U}{wncy}{m}{n}
\DeclareMathSymbol{\Sha}{\mathord}{mcy}{"58}

\title[Jordan Groups]{Jordan constants of quaternion algebras over number fields and simple abelian surfaces over fields of positive characteristic}
\author[WonTae Hwang]{WonTae Hwang$^*$}
\address{School of Mathematics, Korea Institute for Advanced Study, Hoegiro 85, Seoul, South Korea}
\email{hwangwon@kias.re.kr}

\begin{document}



\maketitle

\begin{abstract}
We compute and provide a detailed description on the Jordan constants of the multiplicative subgroup of quaternion algebras over number fields of small degree. As an application, we determine the Jordan constants of the multiplicative subgroup of the endomorphism algebras of simple abelian surfaces over fields of positive characteristic.    \\

\end{abstract}

\section{Introduction}
Let $G$ be a group. Then the notion of $G$ being a \emph{Jordan group}\footnote{It was named after a French mathematician Camille Jordan, and the classical Jordan's theorem on complex general linear groups provides somewhat fundamental examples of Jordan groups.} was introduced by Popov in \cite{Popov(2010)}. If $G$ is a Jordan group, then we can also consider the Jordan constant of $G$ and denote it by $J_G.$ (For the precise definition, see Definition \ref{Jordan} below.) In view of the classical result of Jordan, saying that the general linear group $\textrm{GL}_n(k)$ for an integer $n \geq 1$ and an algebraically closed field $k$ of characteristic zero is a Jordan group, we can deduce that every affine algebraic group is also a Jordan group. Other interesting examples arise in connection with algebraic geometry. Let $X$ be an irreducible algebraic variety over an algebraically closed field $k$ of characteristic zero, and let $G=\textrm{Bir}(X)$, the group of birational automorphisms of $X$. We might ask whether $G$ is a Jordan group or not. If $X$ is a smooth projective rational variety of dimension $n$, then $G=\textrm{Cr}_n(k),$ the Cremona group of degree $n$ over $k,$ and it was known that $\textrm{Cr}_n (k)$ is a Jordan group for $n=1,2,3.$ More recently, Yuri Prokhorov and Constantin Shramov \cite{PS(2016)} proved a nice result that $\textrm{Cr}_n$ is indeed a Jordan group for arbitrary $n>3$, too, modulo the Borisov-Alexeev-Borisov conjecture, which was also recently proved by Caucher Birkar \cite{CB2016}, and they \cite{PS(2017)} computed the Jordan constants of the Cremona groups of rank $2$ and $3.$ Also, Sheng Meng and De-Qi Zhang \cite{MD(2018)} proved that the full automorphism group of any projective variety over $k$ is a Jordan group, using algebraic group theoretic argument. On the other hand, if the dimension of an irreducible variety $X$ over $k$ is small enough in the sense that $\textrm{dim}~X \leq 2,$ then it is known in \cite[Theorem 2.32]{Popov(2010)} that the group $G:= \textrm{Bir}(X)$ is a Jordan group if and only if $X$ is not birationally isomorphic to $\P^1_k \times E$ where $E$ is an elliptic curve over $k.$ In addition, we have a similar but slightly more complicated result for the case of $X$ being a threefold \cite[Theorem 1.8]{PS(2018)}. \\
\indent In this paper, we consider the case of $G$ being the multiplicative subgroup of a quaternion division algebra over a number field of small degree to prove that $G$ is a Jordan group, and we also explicitly compute the Jordan constant $J_G$ of $G$. This gives another context that connects the theory of Jordan groups, algebraic geometry and algebraic number theory. In this aspect, one of our main results is the following
\begin{theorem}\label{main theorem1}
	Let $\mathcal{D}$ be a quaternion algebra over $\Q$ (resp.\ over a quadratic number field). Then $\mathcal{D}^{\times}$ is a Jordan group and $J_{\mathcal{D}^{\times}} \in \{1,2,12\}$ (resp.\ $J_{\mathcal{D}^{\times}} \in \{1,2,12,24,60\}$). Furthermore, we have an explicit description on the relation between $\mathcal{D}$ and $J_{\mathcal{D}^{\times}}$ in terms of the center of $\mathcal{D}$ and the ramification set of $\mathcal{D}.$
\end{theorem}
The proof of Theorem \ref{main theorem1} requires a variety of knowledge including theory of quaternion algebras over number fields and the ramification theory of number fields, combined carefully. For more details, see Theorems \ref{over Q main thm} and \ref{over quad main thm} below. \\
\indent As an application of the above result, we compute the Jordan constants of the multiplicative subgroup of the endomorphism algebras of simple abelian surfaces over fields of positive characteristic. To this aim, we recall: let $X$ be an abelian surface over a field $k.$ We denote the endomorphism ring of $X$ over $k$ by $\textrm{End}_k(X)$. It is a free $\Z$-module of rank $\leq 16.$ Because it is usually harder to deal with the ring $\textrm{End}_k(X),$ we also consider $\textrm{End}^0_k(X):=\textrm{End}_k(X) \otimes_{\Z} \Q.$ This $\Q$-algebra $\textrm{End}_k^0(X)$ is called the endomorphism algebra of $X$ over $k.$ Then $\textrm{End}_k^0(X)$ is a finite dimensional semisimple algebra over $\Q$ with $\textrm{dim}_{\Q} \textrm{End}_k^0(X) \leq 16.$ Moreover, if $X$ is $k$-simple, then $\textrm{End}_k^0(X)$ is a division algebra over $\Q$. Then another main result of this paper is summarized as follows:
\begin{theorem}\label{main theorem2}
Let $n$ be an integer. Then we have: \\
(a) There exists a simple abelian surface $X$ over a finite field $k:=\F_q$ with $q=p^a$ for some prime $p>0$ and an integer $a \geq 1$ such that $J_{D^{\times}}=n,$ where $D:=\textrm{End}_k^0(X)$ if and only if $n \in \{1,2,12,24,60\}.$
(b) There exists a simple abelian surface $X$ over some algebraically closed field $k$ of characteristic $p > 0$ such that $J_{D^{\times}}=n,$ where $D:=\textrm{End}_k^0(X)$ if and only if $n=1.$
\end{theorem}
The proof of Theorem \ref{main theorem2} requires a concrete knowledge of the theory of abelian varieties over fields of positive characteristic, together with Theorem \ref{main theorem1}. More precisely, the proof of one direction of part (a) is obtained by constructing desired simple abelian surfaces concretely, while the proof of the same direction of part (b) is achieved by adopting a slightly more abstract argument.
For more details, see Corollary \ref{simple cor} and Theorems \ref{realized finite}, \ref{realized closed} below. \\ 

\indent Another thing that might be interesting is the following observation, which is related to Theorem \ref{main theorem2}-(a) above:  
\begin{theorem}\label{main theorem3}
Let $S$ be the set of pairs $(n,p)$ of an integer $n \geq 1$ and a prime $p$ with the property that there is a simple abelian surface $X$ over $k:=\F_{q}$ with $q=p^a$ ($a \geq 1$) and $J_{D^{\times}}=n$, where $D:=\textrm{End}_k^0(X).$ If we choose an element $(n,p)$ randomly from $S,$ then it is most probable that $n =12$. 
\end{theorem}
This can be derived from Corollary \ref{quad cor}. For a slightly more detailed statement, see Remark \ref{dist rem} below. \\


This paper is organized as follows: Section~\ref{prelim} is devoted to several facts which are related to our desired results. Explicitly, we briefly review the theory of Jordan groups ($\S$\ref{jordan gp}), and then, we recall some facts about endomorphism algebras of simple abelian varieties ($\S$\ref{end alg av}), the theorem of Tate ($\S$\ref{thm Tate sec}), Honda-Tate theory ($\S$\ref{thm Honda}), and a result of Waterhouse ($\S$\ref{thm waterhouse}). In Section~\ref{main}, we obtain the desired results (Theorems \ref{main theorem1} and \ref{main theorem2}-(a) above) using the facts that were introduced in the previous sections. Finally, in Section \ref{alg closed case}, we prove another desired result, namely, Theorem \ref{main theorem2}-(b) above. \\

In the sequel, let $\overline{k}$ denote an algebraic closure of a field $k.$ Also, for an integer $n \geq 1,$ we denote the cyclic group of order $n$ by $C_n$.

\section{Preliminaries}\label{prelim}
In this section, we recall some of the facts in the theory of Jordan groups and the general theory of abelian varieties over a field. Our main references are \cite{2}, \cite{8}, and \cite{Popov}.

\subsection{Jordan Groups}\label{jordan gp}
In this section, we review the notion of Jordan groups, which is of our main interest in this paper. Our main reference is a paper of Popov \cite{Popov}.  
\begin{definition}\label{Jordan}
A group $G$ is called a \emph{Jordan group} if there exists an integer $d>0$, depending only on $G$, such that every finite subgroup $H$ of $G$ contains a normal abelian subgroup whose index in $H$ is at most $d.$ The minimal such $d$ is called the \emph{Jordan constant of $G$} and is denoted by $J_{G}$.
\end{definition}
In particular, we note that if every finite subgroup of $G$ is abelian, then $G$ is a Jordan group and $J_G =1.$ The following example illustrates this observation:

\begin{example}\label{ab exam}
Let $k$ be an algebraically closed field of characteristic zero, and let $X \subseteq \A_k^4$ be the nonsingular hypersurface defined by the equation $x_1^2 x_2 +x_3^2 + x_4^3 +x_1 =0.$ Then in view of \cite[\S2.2.5]{Popov}, every finite subgroup of $\textrm{Aut}(X)$ is cyclic, and hence, we have $J_{\textrm{Aut}(X)}=1.$
\end{example}
In this paper, we are interested in the situation which is somewhat different from that of Example \ref{ab exam} (see Section \ref{main} below). \\
\indent For our later use in Section \ref{main} to compute the Jordan constants of certain infinite groups, we also record the following useful result:
\begin{lemma}\label{fund lem}
Let $G$ be a Jordan group. Then we have
\begin{equation*}
J_G = \sup_{H \leq G} J_H
\end{equation*}  
where the supremum is taken over all finite subgroups $H$ of $G.$
\end{lemma}
\begin{proof}
For convenience, let $\displaystyle d = \sup_{H \leq G} J_H.$ First, we note that if $H$ is a finite subgroup of $G$, then $H$ is also Jordan and $J_H \leq J_G.$ Hence it clearly follows that $d \leq J_G.$ To prove the reverse inequality, suppose on the contrary that $d< J_G.$ Then by the minimality of $J_G,$ we can see that there is a finite subgroup $H^{\prime}$ of $G$ such that $H^{\prime}$ contains no abelian normal subgroups of index $\leq d.$ Put $d^{\prime}:= J_{H^{\prime}}.$ Then since $d^{\prime} \leq d,$ it follows from the definition of $J_{H^{\prime}}$ that every finite subgroup of $H^{\prime}$ contains an abelian normal subgroup of index $\leq d^{\prime} \leq d.$ In particular, $H^{\prime}$ contains an abelian normal subgroup of index $\leq d$, which is a contradiction. Hence we get $J_G \leq d.$

\indent This completes the proof.
\end{proof}
This lemma is especially helpful if we have a great deal of information on the set of all the finite subgroups of an infinite group $G.$ For an explicit use of Lemma \ref{fund lem}, see the proof of Theorem \ref{over Q main thm} below.

\subsection{Endomorphism algebras of simple abelian varieties over a finite field}\label{end alg av}
In this section, we review some general facts about the endomorphism algebras of simple abelian varieties over a finite field. \\
\indent Let $X$ be a simple abelian variety of dimension $g$ over a finite field $k.$ Then it is well known that $\textrm{End}_k^0(X)$ is a division algebra over $\Q$ with $2g \leq \textrm{dim}_{\Q} \textrm{End}_k^0(X) < (2g)^2$ (see Corollary \ref{cor TateEnd0} below). Before giving our first result, we also recall Albert's classification. We choose a polarization $\lambda : X \rightarrow \widehat{X}$ where $\widehat{X}$ denotes the dual abelian variety of $X$. Using the polarization $\lambda,$ we can define an involution, called the \emph{Rosati involution}, $^{\vee}$ on $\textrm{End}_k^0(X).$ (For a more detailed discussion about the Rosati involution, see \cite[\S20]{8}.) In this way, to the pair $(X,\lambda)$ we associate the pair $(D, ^{\vee})$ with $D=\textrm{End}_k^0(X)$ and $^{\vee}$, the Rosati involution on $D$. We know that $D$ is a simple division algebra over $\Q$ of finite dimension and that $^{\vee}$ is a positive involution. Let $K$ be the center of $D$ so that $D$ is a central simple $K$-algebra, and let $K_0 = \{ x \in K~|~x^{\vee} = x \}$ be the subfield of symmetric elements in $K.$ By a theorem of Albert, $D$ (together with $^\vee$) is of one of the following four types: \\
\indent (i) Type I: $K_0 = K=D$ is a totally real field. \\
\indent (ii) Type II: $K_0 = K$ is a totally real field, and $D$ is a quaternion algebra over $K$ with $D\otimes_{K,\sigma} \R \cong M_2(\R)$ for every embedding $\sigma : K \hookrightarrow \R.$  \\
\indent (iii) Type III: $K_0 = K$ is a totally real field, and $D$ is a quaternion algebra over $K$ with $D\otimes_{K,\sigma} \R \cong \mathbb{H}$ for every embedding $\sigma : K \hookrightarrow \R$ (where $\mathbb{H}$ is the Hamiltonian quaternion algebra over $\R$). \\
\indent (iv) Type IV: $K_0 $ is a totally real field, $K$ is a totally imaginary quadratic field extension of $K_0$, and $D$ is a central simple algebra over $K$. \\
\indent Keeping the notations as above, we let
\begin{equation*}
e_0= [K_0 :\Q],~~~e=[K:\Q],~~~\textrm{and}~~~d=[D:K]^{\frac{1}{2}}.
\end{equation*}

As our last preliminary fact of this section, we impose some numerical restrictions on those values $e_0, e,$ and $d$ in the next table, following \cite[\S21]{8}.
\begin{center}
	\begin{tabular}{c c c}
		\hline
		& $\textrm{char}(k)=0$ & $\textrm{char}(k)=p>0$ \\
		\hline
		$\textrm{Type I}$ & $e|g$ & $e|g$  \\
		\hline
		$\textrm{Type II}$ & $2e|g$  & $2e|g$ \\
		\hline
		$\textrm{Type III}$ & $2e|g$ & $e|g$ \\
		\hline
		$\textrm{Type IV}$ & $e_0 d^2 |g$ &$ e_0 d |g $ \\
		\hline
	\end{tabular}
	\vskip 4pt
	\textnormal{Table 1 } \\
	\textnormal{Numerical restrictions on endomorphism algebras.}
\end{center}

If $g=2$ i.e.\ if $X$ is a simple abelian surface over $k,$ then we readily have the following
\begin{lemma}\label{poss end alg}
	Let $X$ be a simple abelian surface over a finite field $k=\F_q$, and let $\lambda : X \rightarrow \widehat{X}$ be a polarization. Then $D:=\textrm{End}_k^0(X)$ (together with the Rosati involution $^\vee$ corresponding to $\lambda$) is of one of the following three types: \\
	(1) $D$ is a totally definite quaternion algebra over either $\Q$ or a real quadratic field; \\
        (2) $D$ is a CM-field of degree $4$; \\
	(3) $D$ is a quaternion division algebra over an imaginary quadratic field.
\end{lemma}
\begin{proof}
First, we recall that $X$ is of CM-type (see Corollary \ref{cor TateEnd0}-(c) below), and hence, either $D$ is of Type III or of Type IV in Albert's classification. We consider each type one by one. \\
(i) Suppose that $D$ is of Type III. By Table 1, either $e=1$ or $e=2,$ and hence, we get (1). \\
(ii) Suppose that $D$ is of Type IV. Then by the fact that $4 \leq \dim_{\Q} D <16$ and Table 1, we have that the pair $(e_0, d)$ is contained in the set $\{(1,2),(2,1)\}.$ If $(e_0,d)=(2,1),$ then we get (2). Finally, if $(e_0,d)=(1,2),$ then we get (3).\\
\indent This completes the proof.
\end{proof}


\subsection{The theorem of Tate}\label{thm Tate sec}
In this section, we recall an important theorem of Tate, and give some interesting consequences of it. \\
\indent Let $k$ be a field and let $l$ be a prime number with $l \ne \textrm{char}(k)$. If $X$ is an abelian variety of dimension $g$ over $k,$ then we can introduce the Tate $l$-module $T_l X$ and the corresponding $\Q_l$-vector space $V_l X :=T_l X \otimes_{\Z_l} \Q_l.$ It is well known that $T_l X$ is a free $\Z_l$-module of rank $2g$ and $V_l X$ is a $2g$-dimensional $\Q_l$-vector space. In \cite{11}, Tate showed the following important result for the case when $k$ is a finite field:

\begin{theorem}\label{thm Tate}
	Let $k$ be a finite field and let $\Gamma = \textrm{Gal}(\overline{k}/k).$ If $l$ is a prime number with $l \ne \textrm{char}(k),$ then we have: \\
	(a) For any abelian variety $X$ over $k,$ the representation
	\begin{equation*}
	\rho_l =\rho_{l,X} : \Gamma \rightarrow \textrm{GL}(V_l X)
	\end{equation*}
	is semisimple. \\
	(b) For any two abelian varieties $X$ and $Y$ over $k,$ the map
	\begin{equation*}
	\Z_l \otimes_{\Z} \textrm{Hom}_k(X,Y) \rightarrow \textrm{Hom}_{\Gamma}(T_l X, T_l Y)
	\end{equation*}
	is an isomorphism.
\end{theorem}

Now, we recall that an abelian variety $X$ over a (finite) field $k$ is called \emph{elementary} if $X$ is $k$-isogenous to a power of a simple abelian variety over $k.$ Then, as an interesting consequence of Theorem \ref{thm Tate}, we have the following fundamental result:
\begin{corollary}\label{cor TateEnd0}
	Let $X$ be an abelian variety of dimension $g$ over a finite field $k.$ Then we have:\\
	(a) The center $Z$ of $\textrm{End}_k^0(X)$ is the subalgebra $\Q[\pi_X]$ where $\pi_X$ denotes the Frobenius endomorphism of $X.$ In particular, $X$ is elementary if and only if $\Q[\pi_X]=\Q(\pi_X)$ is a field, and this occurs if and only if $f_X$ is a power of an irreducible polynomial in $\Q[t]$ where $f_X$ denotes the characteristic polynomial of $\pi_X.$ \\
	(b) Suppose that $X$ is elementary. Let $h=f_{\Q}^{\pi_X}$ be the minimal polynomial of $\pi_X$ over $\Q$. Further, let $d=[\textrm{End}_k^0(X):\Q(\pi_X)]^{\frac{1}{2}}$ and $e=[\Q(\pi_X):\Q].$ Then $de =2g$ and $f_X = h^d.$ \\
	(c) We have $2g \leq \textrm{dim}_{\Q} \textrm{End}^0_k (X) \leq (2g)^2$ and $X$ is of CM-type. \\
	(d) The following conditions are equivalent: \\
	\indent (d-1) $\textrm{dim}_{\Q} \textrm{End}_k^0(X)=2g$; \\
	\indent (d-2) $\textrm{End}_k^0(X)=\Q[\pi_X]$; \\
	\indent (d-3) $\textrm{End}_k^0(X)$ is commutative; \\
	\indent (d-4) $f_X$ has no multiple root. \\
	(e) The following conditions are equivalent: \\
	\indent (e-1) $\textrm{dim}_{\Q} \textrm{End}_k^0(X)=(2g)^2$; \\
	\indent (e-2) $\Q[\pi_X]=\Q$; \\
	\indent (e-3) $f_X$ is a power of a linear polynomial; \\
	\indent (e-4) $\textrm{End}^0_k(X) \cong M_g(D_{p,\infty})$ where $D_{p,\infty}$ is the unique quaternion algebra over $\Q$ that is ramified at $p$ and $\infty$, and split at all other primes; \\
	\indent (e-5) $X$ is supersingular with $\textrm{End}_k(X) = \textrm{End}_{\overline{k}}(X_{\overline{k}})$ where $X_{\overline{k}}=X \times_k \overline{k}$; \\
	\indent (e-6) $X$ is isogenous to $E^g$ for a supersingular elliptic curve $E$ over $k$ all of whose endomorphisms are defined over $k.$
\end{corollary}
\begin{proof}
	For a proof, see \cite[Theorem 2]{11}.
\end{proof}

For a precise description of the structure of the endomorphism algebra of a simple abelian variety $X$, viewed as a simple algebra over its center $\Q[\pi_X]$, we record the following useful result:
\begin{proposition}\label{local inv}
	Let $X$ be a simple abelian variety over a finite field $k=\F_q.$ Let $K=\Q[\pi_X].$ Then we have: \\
	(a) If $\nu$ is a place of $K$, then the local invariant of $\textrm{End}_k^0(X)$ in the Brauer group $\textrm{Br}(K_{\nu})$ is given by
	\begin{equation*}
	\textrm{inv}_{\nu}(\textrm{End}_k^0(X))=\begin{cases} 0 & \mbox{if $\nu$ is a finite place not above $p$}; \\ \frac{\textrm{ord}_{\nu}(\pi_X)}{\textrm{ord}_{\nu}(q)} \cdot [K_{\nu}:\Q_p] & \mbox{if $\nu$ is a place above $p$}; \\ \frac{1}{2} & \mbox{if $\nu$ is a real place of $K$}; \\ 0 & \mbox{if $\nu$ is a complex place of $K$}. \end{cases}
	\end{equation*}
	(b) If $d$ is the degree of the division algebra $D:=\textrm{End}_k^0(X)$ over its center $K$ (so that $d=[D:K]^{\frac{1}{2}}$ and $f_X = (f_{\Q}^{\pi_X})^d$), then $d$ is the least common denominator of the local invariants $\textrm{inv}_{\nu}(D).$
\end{proposition}
\begin{proof}
	(a) For a proof, see \cite[Corollary 16.30]{2}. \\
	(b) For a proof, see \cite[Corollary 16.32]{2}.
\end{proof}

\subsection{Abelian varieties up to isogeny and Weil numbers: Honda-Tate theory}\label{thm Honda}
In this section, we recall an important theorem of Honda and Tate. Throughout this section, let $q=p^a$ for some prime $p$ and an integer $a \geq 1.$ To achieve our goal, we first give the following
\begin{definition}\label{qWeil Def}
	(a) A \emph{$q$-Weil number} is an algebraic integer $\pi$ such that $| \iota(\pi) | = \sqrt{q}$ for all embeddings $\iota : \Q[\pi] \hookrightarrow \C.$ \\
	(b) Two $q$-Weil numbers $\pi$ and $\pi^{\prime}$ are said to be \emph{conjugate} if they have the same minimal polynomial over $\Q,$ or equivalently, there is an isomorphism $\Q[\pi] \rightarrow \Q[\pi^{\prime}]$ sending $\pi$ to $\pi^{\prime}.$
\end{definition}

Regarding $q$-Weil numbers, the following facts are well-known:
\begin{remark}\label{qWeil rem}
	Let $X$ and $Y$ be simple abelian varieties over a finite field $k=\F_q.$ Then we have \\
	(i) The Frobenius endomorphism $\pi_X$ is a $q$-Weil number. \\
	(ii) $X$ and $Y$ are $k$-isogenous if and only if $\pi_X$ and $\pi_Y$ are conjugate.
\end{remark}

Now, we introduce our main result of this section:

\begin{theorem}\label{thm HondaTata}
	For every $q$-Weil number $\pi$, there exists a simple abelian variety $X$ over $\F_q$ such that $\pi_X$ is conjugate to $\pi$. Moreover, we have a bijection between the set of isogeny classes of simple abelian varieties over $\F_q$ and the set of conjugacy classes of $q$-Weil numbers given by $X \mapsto \pi_X$.
\end{theorem}
The inverse of the map $X \mapsto \pi_X$ associates to a $q$-Weil number $\pi$ a simple abelian variety $X$ over $\F_q$ such that $f_X$ is a power of the minimal polynomial $f_{\Q}^{\pi}$ of $\pi$ over $\Q.$
\begin{proof}
	For a proof, see \cite[Main Theorem]{4} or \cite[\S16.5]{2}.
\end{proof}

\subsection{Isomorphism classes contained in an isogeny class}\label{thm waterhouse}
In this section, we will give a useful result of Waterhouse~\cite{12}. Throughout this section, let $k=\F_q$ with $q=p^a$ for some prime $p$ and an integer $a \geq 1.$ \\
\indent Let $X$ be an abelian variety over $k.$ Then $\textrm{End}_k(X)$ is a $\Z$-order in $\textrm{End}_k^0(X)$ containing $\pi_X$ and $q/\pi_X.$ If a ring $R$ is the endomorphism ring of an abelian variety, then we may consider a left ideal of $R$, and give the following
\begin{definition}\label{Roccur}
	Let $X$ be an abelian variety over $k$ with $R:=\textrm{End}_k(X),$ and let $I$ be a left ideal of $R$, which contains an isogeny. \\
	(a) We define $H(X,I)$ to be the intersection of the kernels of all elements of $I$. This is a finite subgroup scheme of $X.$ \\
	(b) We define $X_I$ to be the quotient of $X$ by $H(X,I)$ i.e.\ $X_I=X/H(X,I).$ This is an abelian variety over $k$ that is $k$-isogenous to $X.$
\end{definition}

Now, we introduce our main result of this section, which plays an important role:
\begin{proposition}\label{endposs}
	Let $X$ be an abelian variety over $k$ with $R:=\textrm{End}_k(X)$, and let $I$ be a left ideal of $R$, which contains an isogeny. Also, let $D=\textrm{End}_k^0(X).$ Then we have: \\
	(a) $\textrm{End}_k(X_I)$ contains $O_r (I):=\{x \in \textrm{End}_k^0(X)~|~I x \subseteq I\}$, the right order of $I,$ and equals it if $I$ is a kernel ideal\footnote{Recall that $I$ is a kernel ideal of $R$ if $I=\{\alpha \in R~|~\alpha H(X,I)=0 \}$.}. \\
	(b) Every maximal order in $D$ occurs as the endomorphism ring of an abelian variety in the isogeny class of $X.$
\end{proposition}
\begin{proof}
	(a) For a proof, see \cite[Lemma 16.56]{2} or \cite[Proposition 3.9]{12}. \\
	(b) For a proof, see \cite[Theorem 3.13]{12}.
\end{proof}

\section{Main Result}\label{main}
Throughout this section, let $\mathcal{D}$ be a quaternion division algebra over an algebraic number field $K$ with $[K:\Q] \leq 2$, unless otherwise specified. Then the multiplicative subgroup $\mathcal{D}^{\times}$ of $\mathcal{D}$ is infinite, containing $K^{\times}$. Let $\textrm{Ram}(\mathcal{D})$ be the set of all primes of $K$ at which $\mathcal{D}$ is ramified. Then the following is fundamental: 
\begin{remark}\label{fund rem}
Let $\mathcal{D}$ and $\mathcal{D}^{\prime}$ be quaternion division algebras over an algebraic number field $K$ with $[K:\Q] \leq 2.$ Then we have:\\
(i) The set $\textrm{Ram}(\mathcal{D})$ is finite and the cardinality of $\textrm{Ram}(\mathcal{D})$ is even and positive. \\
(ii) $\mathcal{D} \cong \mathcal{D}^{\prime}$ if and only if $\textrm{Ram}(\mathcal{D})=\textrm{Ram}(\mathcal{D}^{\prime}).$
\end{remark}
\begin{example}
Let $p \geq 2$ be a prime. In the sequel, we let $D_{p,\infty}$ denote the unique quaternion division algebra over $\Q$ which is ramified precisely at the primes $p$ and $\infty.$ In other words, we have $\textrm{Ram}(D_{p,\infty})=\{p, \infty\}.$ In particular, if $p$ and $p^{\prime}$ are two distinct rational primes, then $D_{p,\infty}$ is not isomorphic to $D_{p^{\prime},\infty}.$
\end{example}

Now, we consider the case when $K=\Q$, in more details. 
\begin{lemma}\label{over Q lem}
Let $\mathcal{D}$ be a quaternion division algebra over $\Q.$ Then the group $\mathcal{D}^{\times}$ is Jordan, and we have $J_{\mathcal{D}^{\times}} \leq 60.$
\end{lemma}
\begin{proof}
Since $\mathcal{D} \otimes_{\Q} \C \cong M_2(\C)$ so that $\mathcal{D}^{\times} \leq \textrm{GL}_2 (\C)$, it follows from \cite[Theorems 1 and 3-(1)]{Popov} that $\mathcal{D}^{\times}$ is Jordan and $J_{\mathcal{D}^{\times}} \leq J_{\textrm{GL}_2(\C)}=60$.
\end{proof}

Next, we look into the case when $[K:\Q] = 2.$ If $K$ is assumed to be totally real, then we have

\begin{lemma}\label{real quad lem}
Let $\mathcal{D}$ be a quaternion division algebra over a real quadratic field $K.$ Then the group $J_{\mathcal{D}^{\times}}$ is Jordan, and we have $J_{\mathcal{D}^{\times}} \leq 3600.$
\end{lemma}
\begin{proof}
Note that if $\mathcal{D}$ is totally definite (resp.\ indefinite), then we have $\mathcal{D} \otimes_{\Q} \mathbb{R} \cong \mathbb{H} \oplus \mathbb{H}$ (resp.\ $\mathcal{D} \otimes_{\Q} \mathbb{R} \cong M_2 (\mathbb{R}) \oplus M_2(\mathbb{R}))$ so that $\mathcal{D} \otimes_{\Q} \C \cong M_2 (\C) \oplus M_2 (\C)$ in both cases. 
Then since $\mathcal{D}^{\times} \leq \textrm{GL}_2(\C) \times \textrm{GL}_2(\C),$ it follows from \cite[Theorems 1 and 3-(1),(2)]{Popov} that $\mathcal{D}^{\times}$ is Jordan and $J_{\mathcal{D}^{\times}} \leq J_{\textrm{GL}_2(\C)}^2 = 3600.$ 
\end{proof}

Similarly,
\begin{lemma}\label{imag quad lem}
Let $\mathcal{D}$ be a quaternion division algebra over an imaginary quadratic field $K.$ Then the group $J_{\mathcal{D}^{\times}}$ is Jordan, and we have $J_{\mathcal{D}^{\times}} \leq 60.$
\end{lemma}
\begin{proof}
Note that we have $\mathcal{D} \otimes_{\Q} \R \cong M_2 (\C)$, and then, since $\mathcal{D}^{\times} \leq \textrm{GL}_2(\C),$ it follows from \cite[Theorems 1 and 3-(1)]{Popov} that $\mathcal{D}^{\times}$ is Jordan, and $J_{\mathcal{D}^{\times}} \leq J_{\textrm{GL}_2(\C)} =60.$
\end{proof}



Now, we would like to describe the exact value of $J_{\mathcal{D}^{\times}}$, and see that all the upper bounds in Lemmas \ref{over Q lem}, \ref{real quad lem}, and \ref{imag quad lem} are not sharp. To this aim, we first introduce two lemmas:
\begin{lemma}\label{dic lem}
For any integer $n \geq 2,$ we have 
\begin{equation*}
J_{\textrm{Dic}_{4n}}=2
\end{equation*}
where $\textrm{Dic}_{4n}$ denotes the dicyclic group of order $4n.$ 
\end{lemma}
\begin{proof}
Let $G=\textrm{Dic}_{4n}$ $(n \geq 2)$ and $H$ a subgroup of $G.$ Note that $G$ admits a 2-dimensional faithful irreducible symplectic representation, and hence, $H$ admits a faithful symplectic representation (which is induced from that of $G$). If this representation is irreducible, then $H$ must be a dicyclic group, too. If it is reducible, then $H$ must be cyclic (so that it does not contribute to the Jordan constant of $G$). \\
\indent Now, since $G$ contains a $C_{2n}$ as an abelian normal subgroup, we have $J_G \geq [G:C_{2n}]=2.$ If $J_G >2,$ then by the minimality of $J_G,$ it follows that there is a subgroup $H$ of $G$ such that $H$ contains no abelian normal subgroups of index $\leq 2.$ But then, since $H$ is either a dicyclic group or a cyclic group by the above observation, this is a contradiction. Hence we can see that $J_G=2,$ as desired. \\
\indent This completes the proof.
\end{proof}

\begin{lemma}\label{binary lem}
Let $G \in \{\mathfrak{T}^*, \mathfrak{O}^*, \mathfrak{I}^*\}.$ Then we have
\begin{equation*}
	J_{G}=\begin{cases} 12 & \mbox{if $G=\mathfrak{T}^*$}; \\ 24 & \mbox{if $G=\mathfrak{O}^*$}; \\ 60 & \mbox{if $G=\mathfrak{I}^*$} \end{cases}
	\end{equation*}
where $\mathfrak{T}^*$ (resp.\ $\mathfrak{O}^*$, resp.\ $\mathfrak{I}^*$) denotes the binary tetrahedral group (resp.\ binary octahedral group, resp.\ binary icosahedral group). 
\end{lemma}
\begin{proof}
Let $G=\mathfrak{T}^*$. Then all the abelian subgroups of $G$ are $C_n$ for $n \in \{1,2,3,4,6\},$ among which the normal ones are $C_2$ and the trivial group. Hence it follows that $J_{G} \geq [G:C_2]=12.$ On the other hand, if $H$ is an arbitrary proper subgroup of $G,$ then $|H| \leq 8,$ (where the equality holds if we take $H=Q_8$), and hence, we can see that $H$ contains an abelian normal subgroup of index $\leq 8 <12$. Then it follows from the minimality of the Jordan constant that $J_{G} \leq 12.$ Thus we get $J_G = 12,$ as desired.  \\
\indent Similarly, for $G=\mathfrak{O}^*$, among the normal subgroups $\mathfrak{O}^*, \textrm{SL}_2(\F_3), Q_8, C_2,$ and the trivial group of $\mathfrak{O}^*$, the only abelian groups are $C_2$ and the trivial group, and hence, it follows that $J_G  \geq  [G:C_2] = 24$. On the other hand, if $H$ is an arbitrary proper subgroup of $G,$ then $|H| \leq 24$, (where the equality holds if we take $H=\textrm{SL}_2(\F_3)$), and hence, we can see that $H$ contains an abelian normal subgroup of index $\leq 24$. Then it follows from the minimality of the Jordan constant that $J_G \leq 24$, and hence, we get $J_G=24$, as desired. \\
\indent Finally, if $G=\mathfrak{I}^*$, then among the normal subgroups $\mathfrak{I}^*, C_2,$ and the trivial group of $\mathfrak{I}^*$, the only abelian groups are $C_2$ and the trivial group, and hence, it follows that $J_G  \geq  [G:C_2]= 60$. Let $d=J_G$ and assume that $d>60.$ Then by definition, there is a subgroup $H$ of $G$ such that $H$ contains no abelian normal subgroups of index $\leq 60.$ But then since $|H| \leq 24$ (by looking at the list of all subgroups of $G$), this is a contradiction. Hence we can see that $J_G=d=60.$  \\
\indent This completes the proof.  
\end{proof}
Using Lemmas \ref{fund lem}, \ref{dic lem}, and \ref{binary lem} above, we obtain:
\begin{theorem}\label{over Q main thm}
Let $\mathcal{D}$ be a quaternion division algebra over $\Q$ and let $R_{\mathcal{D}}=\textrm{Ram}(\mathcal{D}).$ Then we have
\begin{equation*}
	J_{\mathcal{D}^{\times}}=\begin{cases} 12 & \mbox{if and only if $R_{\mathcal{D}}=\{2,\infty\}$}; \\ 2 & \mbox{if and only if $R_{\mathcal{D}}=\{3,\infty\}$}; \\ 1 &  \mbox{otherwise}. \end{cases}
	\end{equation*}
\end{theorem}
\begin{proof} 
Let $G$ be a finite subgroup of $\mathcal{D}^{\times}.$ Then in light of \cite[Theorem 5.10]{Hwang(2018)}, together with dimension counting over $\Q$, we can see that $G$ is either one of the following three groups $\mathfrak{T}^*, \textrm{Dic}_{12}, Q_8$ or is a cyclic group. Note that $J_{\mathfrak{T}^*}=12, J_{\textrm{Dic}_{12}}=J_{Q_8}=2$ by Lemmas \ref{dic lem}, \ref{binary lem}, and the Jordan constant of cyclic groups is equal to $1.$ \\
\indent  Now, we observe that $R_{\mathcal{D}}=\{2,\infty\}$ if and only if $\mathcal{D}\cong D_{2,\infty}.$ In light of \cite[Theorem 9]{1} and \cite[Theorem 6.1]{Nebe(1998)}, we also have that $\mathfrak{T}^*$ is an absolutely irreducible maximal finite subgroup of $\mathcal{D}^{\times}$ if and only if $\mathcal{D} \cong D_{2,\infty}.$ Now, suppose that $R_{\mathcal{D}}=\{2,\infty\}.$ Then since every finite subgroup of $\mathcal{D}^{\times} \cong D_{2,\infty}^{\times}$ is a subgroup of $\mathfrak{T}^*,$ it follows from Lemma \ref{fund lem} that $J_{\mathcal{D}^{\times}}= J_{\mathfrak{T}^*}=12.$ Conversely, if $J_{\mathcal{D}^{\times}}=12,$ then by Lemmas \ref{fund lem} and \ref{binary lem}, we know that $\mathfrak{T}^* \leq \mathcal{D}^{\times},$ and hence, $\mathcal{D} \cong D_{2,\infty}.$ In a similar fashion, we can see that $J_{\mathcal{D}^{\times}}=2$ if and only if $\textrm{Dic}_{12} \leq \mathcal{D}^{\times},$ which is also equivalent to the fact that $\mathcal{D} \cong D_{3,\infty}.$\footnote{Here, we implicitly use a fact that if $Q_8 \leq \mathcal{D}^{\times},$ then $\mathcal{D} \cong D_{2,\infty}$ by \cite[Theorem 9]{1}, in which case, we have seen that $J_{\mathcal{D}^{\times}}=12$.} The last item follows from the observation that, in this case, every finite subgroup of $\mathcal{D}^{\times}$ is cyclic. \\
\indent This completes the proof. 
\end{proof}


The following is a restatement of the above theorem in a special case, that is related to our application: 
\begin{corollary}\label{Q cor}
Assume that $\mathcal{D}=D_{p,\infty}$ for some prime $p \geq 2$. Then we have
\begin{equation*}
	J_{\mathcal{D}^{\times}}=\begin{cases} 12 & \mbox{if and only if $p=2$}; \\ 2 & \mbox{if and only if $p=3$}; \\ 1 & \mbox{if and only if $p \geq 5$} .\end{cases}
	\end{equation*}
\end{corollary}
\begin{proof}
We recall that $D_{p,\infty}$ is ramified exactly at the primes $p$ and $\infty,$ and hence, the desired result follows immediately from Theorem \ref{over Q main thm}. 
\end{proof}

The main difference between Theorem \ref{over Q main thm} and Corollary \ref{Q cor} is that the $R_{\mathcal{D}}$ in the last item of Theorem \ref{over Q main thm} can be any set of places of $\Q$, whose cardinality is even and positive (see Remark \ref{fund rem} above). 
 
\begin{remark}
In fact, if $p \equiv 11 \pmod{12},$ then both of $C_4$ and $C_6$ are subgroups of $\mathcal{D}^{\times}:=D_{p,\infty}^{\times}$, and if $p \equiv 5 \pmod{12},$ then $C_6$ is a subgroup of $\mathcal{D}^{\times}$ while $C_4$ is not. Also, if $p \equiv 7~\pmod{12},$ then $C_4$ is a subgroup of $\mathcal{D}^{\times}$ while $C_6$ is not, and if $p \equiv 1 \pmod{12},$ then neither $C_4$ nor $C_6$ is a subgroup of $\mathcal{D}^{\times}.$
\end{remark}

Similarly, we also have
\begin{theorem}\label{over quad main thm}
Let $\mathcal{D}$ be a quaternion division algebra over a quadratic number field $K$ and let $R_{\mathcal{D}}=\textrm{Ram}(\mathcal{D}).$ Then we have
\begin{equation*}
	J_{\mathcal{D}^{\times}}=\begin{cases} 60 & \mbox{if and only if $K=\Q(\sqrt{5})$ and $R_{\mathcal{D}}=R_{\infty}$}; \\ 24 & \mbox{if and only if $K=\Q(\sqrt{2})$ and $R_{\mathcal{D}}=R_{\infty}$}; \\ 12 & \mbox{iff $\mathcal{D}=D_{2,\infty} \otimes_{\Q} K$ where $K=\Q(\sqrt{d})$ with $d \ne 2,5$, or}\\ & \mbox{$\mathcal{D}=D_{2,\infty} \otimes_{\Q} K$ where $K=\Q(\sqrt{-d})$ with $d \equiv 7 \pmod{8}$} ; \\  2 & \mbox{iff $\mathcal{D}=D_{3,\infty} \otimes_{\Q} K$ where $K = \Q(\sqrt{d})$ with $d \equiv 9,17 \pmod{24}$ or $d \equiv 1 \pmod{3}$, or} \\  &  \mbox{$\mathcal{D}=D_{3,\infty} \otimes_{\Q} K$ where $K = \Q(\sqrt{-d})$ with $d \equiv 2 \pmod{3}$;} \\ 1 & \mbox{otherwise}  \end{cases}
	\end{equation*}
where $d$ is, a priori, meant to be a positive, square-free integer in all the items above, and $R_{\infty}$ denotes the set of all infinite places of (the corresponding) $K$. 
\end{theorem}
\begin{proof}
Let $G$ be a finite subgroup of $\mathcal{D}^{\times}.$ Then in light of \cite[Theorem 5.10]{Hwang(2018)}, we can see that $G$ is either one of the following three groups $\mathfrak{T}^*, \mathfrak{O}^*, \mathfrak{I}^*$ or is a dicyclic group or is a cyclic group. Note that $J_{\mathfrak{T}^*}=12, J_{\mathfrak{O}^*}=24, J_{\mathfrak{I}^*}=60$ by Lemma \ref{binary lem}, the Jordan constant of dicyclic groups equals $2$ by Lemma \ref{dic lem}, and the Jordan constant of cyclic groups is equal to $1.$ \\
\indent  Now, we observe that $K=\Q(\sqrt{5})$ and $R_{\mathcal{D}}=R_{\infty}$ if and only if $\mathcal{D}\cong D_{2,\infty} \otimes_{\Q} \Q(\sqrt{5}).$ In light of \cite[Theorem 9]{1} and \cite[Theorem 6.1]{Nebe(1998)}, we also have that $\mathfrak{I}^*$ is an absolutely irreducible maximal finite subgroup of $\mathcal{D}^{\times}$ if and only if $\mathcal{D} \cong D_{2,\infty} \otimes_{\Q} \Q(\sqrt{5}).$ Now, suppose that $K=\Q(\sqrt{5})$ and $R_{\mathcal{D}}=R_{\infty}$. Then since every finite subgroup of $\mathcal{D}^{\times} \cong (D_{2,\infty} \otimes_{\Q} \Q(\sqrt{5}))^{\times}$ is a subgroup of $\mathfrak{I}^*,$ it follows from Lemma \ref{fund lem} that $J_{\mathcal{D}^{\times}}= J_{\mathfrak{I}^*}=60.$ Conversely, if $J_{\mathcal{D}^{\times}}=60,$ then by Lemmas \ref{fund lem} and \ref{binary lem}, we know that $\mathfrak{I}^* \leq \mathcal{D}^{\times},$ and hence, $\mathcal{D} \cong D_{2,\infty} \otimes_{\Q} \Q(\sqrt{5}).$ In a similar fashion, we can see that $J_{\mathcal{D}^{\times}}=24$ if and only if $\mathfrak{O}^* \leq \mathcal{D}^{\times},$ which is also equivalent to the fact that $K=\Q(\sqrt{2})$ and $R_{\mathcal{D}}=R_{\infty}$. For $J_{\mathcal{D}^{\times}}=12,$ we note that $J_{\mathcal{D}^{\times}}=12$ if and only if $\mathfrak{T}^* \leq \mathcal{D}^{\times}$ and $\mathfrak{O}^*, \mathfrak{I}^*$ are not subgroups of $\mathcal{D}^{\times},$ which is also equivalent to the fact that $\mathcal{D} \cong D_{2,\infty} \otimes_{\Q} K$ (in view of \cite[Theorem 9]{1}) that is a division algebra and the conditions for first two items above do not hold. Now, we examine the latter condition further. If $K=\Q(\sqrt{d})$ is a real quadratic field, then $\mathcal{D} \cong D_{2,\infty} \otimes_{\Q} K$ is always a division algebra so that it suffices to exclude the cases of $d=2$ and $d=5.$ If $K=\Q(\sqrt{-d})$ is an imaginary quadratic field, then it suffices to find the condition of $\mathcal{D} \cong D_{2,\infty} \otimes_{\Q} K$ being a division algebra, which is equivalent to the fact that $2$ splits completely in $\Q(\sqrt{-d})$, and the latter condition is equivalent to $d \equiv 7 \pmod{8}.$ Similarly, we can see that $J_{\mathcal{D}^{\times}}=2$ if and only if $\textrm{Dic}_{12} \leq \mathcal{D}^{\times},$\footnote{We need to note that there are, a priori, other dicyclic groups, namely, $Q_8, \textrm{Dic}_{16}, \textrm{Dic}_{20}$, and $\textrm{Dic}_{24}$, that can be embedded in $\mathcal{D}^{\times}$ so that we can have $J_{\mathcal{D}^{\times}}=2.$ However, it can be shown that we can rule out these dicyclic groups as follows: as an example, we give an argument for the group $\textrm{Dic}_{24}.$ If $\textrm{Dic}_{24} \leq \mathcal{D}^{\times},$ then since the 2-Sylow subgroup of $\textrm{Dic}_{24}$ is a $Q_8,$ we can see that $\mathcal{D} \cong D_{2,\infty} \otimes_{\Q} K$ for some $K$ by \cite[Theorem 9]{1}. Then it follows from \cite[Theorem 6.1]{Nebe(1998)} that $K=\Q(\sqrt{3})$ and $R_{\mathcal{D}}=R_{\infty}$, and hence, we can see that $J_{\mathcal{D}^{\times}}=12$ by the third item of this theorem (having $\mathcal{D} \cong D_{2,\infty} \otimes_{\Q} \Q(\sqrt{3})$). A similar argument applies for other groups, too.}    while $\mathfrak{T}^*$ is not a subgroup of $\mathcal{D}^{\times},$ which is also equivalent to saying that $\mathcal{D} \cong D_{3,\infty} \otimes_{\Q} K$ that is a division algebra, and $\mathcal{D} \not \cong D_{2,\infty} \otimes_{\Q} K.$ If $K=\Q(\sqrt{d})$ is a real quadratic field, then it suffices to consider the last condition, and we can see (by a careful analysis) that $\mathcal{D} \not \cong D_{2,\infty} \otimes_{\Q} K$ if and only if $d \equiv 17 \pmod{24}$ (in which case, $3$ is inert and $2$ splits completely in $\Q(\sqrt{d})$) or $d \equiv 1 \pmod{3}$ (in which case, $3$ splits completely in $\Q(\sqrt{d})$ so that the ramification behavior of $2$ does not matter), or $d \equiv 9 \pmod{24}$ (in which case, $3$ is ramified and $2$ splits completely in $\Q(\sqrt{d}))$. If $K=\Q(\sqrt{-d})$ is an imaginary quadratic field, then it suffices to find the condition of $\mathcal{D} \cong D_{3,\infty} \otimes_{\Q} K$ being a division algebra, which is equivalent to the fact that $3$ splits completely in $\Q(\sqrt{-d})$, and the latter condition is equivalent to $d \equiv 2 \pmod{3}$. Finally, the last item follows from the observation that, in this case, every finite subgroup of $\mathcal{D}^{\times}$ is cyclic (in view of \cite[Theorem 5.10]{Hwang(2018)}). \\
\indent This completes the proof. 
\end{proof}

\begin{remark}
To give a slightly more detailed description on the case when $J_{\mathcal{D}^{\times}}=1,$ we consider the following three cases: \\

\noindent (Case I) $K=\Q(\sqrt{-1})$ and $R_{\mathcal{D}} \subseteq \{\mathfrak{p} \in \textrm{Pl}(K)~|~\mathfrak{p}~\textrm{lies over}~2~\textrm{or}~3~\textrm{or}~p \equiv 5 \pmod{12}\}$ or;\\
\indent \indent \indent $K=\Q(\sqrt{-3})$ and $R_{\mathcal{D}} \subseteq \{\mathfrak{p} \in \textrm{Pl}(K)~|~\mathfrak{p}~\textrm{lies over}~2~\textrm{or}~3~\textrm{or}~p \equiv 7 \pmod{12}\}$ or;\\
\indent \indent \indent $K=\Q(\sqrt{3})$ and $R_{\mathcal{D}} \subseteq \{\mathfrak{p} \in \textrm{Pl}(K)~|~\mathfrak{p}~\textrm{lies over}~2~\textrm{or}~3~\textrm{or}~p \equiv 11 \pmod{12}~\textrm{or}~\infty\}$, $R_{\mathcal{D}} \ne R_{\infty}$. \\
(Case II) $K=\Q(\sqrt{5})$ and $R_{\mathcal{D}} \subseteq \{\mathfrak{p} \in \textrm{Pl}(K)~|~\mathfrak{p}~\textrm{lies over}~2~\textrm{or}~5~\textrm{or}~p \equiv 3,7,9 \pmod{10}~\textrm{or}~\infty \}, R_{\mathcal{D}} \ne R_{\infty}$.\\
(Case III) $K=\Q(\sqrt{-1})$ and $R_{\mathcal{D}} \subseteq \{\mathfrak{p} \in \textrm{Pl}(K)~|~\mathfrak{p}~\textrm{lies over}~2~\textrm{or}~p \equiv 5 \pmod{8}\}$, $R_{\mathcal{D}} \not\subseteq \{\mathfrak{p} \in \textrm{Pl}(K)~|~\mathfrak{p}~\textrm{lies over}~2~\textrm{or}~3~\textrm{or}~p \equiv 5 \pmod{12}\}$ or;\\
\indent \indent \indent $K=\Q(\sqrt{-2})$ and $R_{\mathcal{D}} \subseteq \{\mathfrak{p} \in \textrm{Pl}(K)~|~\mathfrak{p}~\textrm{lies over}~2 ~\textrm{or}~p \equiv 3 \pmod{8}\}, R_{\mathcal{D}} \ne \{\mathfrak{p} \in \textrm{Pl}(K)~|~\mathfrak{p}~\textrm{lies over}~3\}$ or;\\
\indent \indent \indent $K=\Q(\sqrt{2})$ and $R_{\mathcal{D}} \subseteq \{\mathfrak{p} \in \textrm{Pl}(K)~|~\mathfrak{p}~\textrm{lies over}~2~\textrm{or}~p \equiv 7 \pmod{8}~\textrm{or}~\infty\}, R_{\mathcal{D}} \ne R_{\infty}$.\\
\indent (Here, $\textrm{Pl}(K)$ denotes the set of all places of $K.$) \\

If (Case I) holds for $\mathcal{D}$, then we have $C_{12} \leq \mathcal{D}^{\times}$. Similarly, if (Case II) (resp.\ (Case III)) holds for $\mathcal{D},$ then we have $C_{10} \leq \mathcal{D}^{\times}$ (resp.\ $C_8 \leq \mathcal{D}^{\times}$). 
\end{remark}


Now, we further specify the quaternion algebra to work with, and obtain a somewhat interesting result:


\begin{corollary}\label{quad cor}
Assume that $\mathcal{D}=D_{p,\infty} \otimes_{\Q} \Q(\sqrt{p})$ for some prime $p$. Then we have
\begin{equation*}
	J_{\mathcal{D}^{\times}}=\begin{cases} 60 & \mbox{if and only if $p=5$}; \\ 24 & \mbox{if and only if $p=2$}; \\ 12  & \mbox{if and only if $p \equiv 3 \pmod{4}$ or $p>5$ and $p \equiv 5 \pmod{8}$} ; \\ 2 & \mbox{if and only if $p \equiv 17 \pmod{24}$}; \\ 1 & \mbox{otherwise}.\end{cases}
	\end{equation*}
\end{corollary}
\begin{proof}
The cases of $J_{\mathcal{D}^{\times}}=60$ or $24$ follow immediately from Theorem \ref{over quad main thm}. For $J_{\mathcal{D}^{\times}}=12,$ we note that $J_{\mathcal{D}^{\times}}=12$ if and only if $\mathcal{D} =D_{p,\infty} \otimes_{\Q} \Q(\sqrt{p}) \cong D_{2,\infty}\otimes_{\Q} \Q(\sqrt{p})$ with $p \ne 2,5$ by Theorem \ref{over quad main thm}, which, in turn, is equivalent to the fact that $p \equiv 3 \pmod{4}$ or $p>5$ and $p \equiv 5 \pmod{8}.$ Similarly, for $J_{\mathcal{D}^{\times}}=2$, we note that $J_{\mathcal{D}^{\times}}=2$ if and only if $\mathcal{D}=D_{p,\infty} \otimes_{\Q} \Q(\sqrt{p}) \cong D_{3,\infty}\otimes_{\Q} \Q(\sqrt{p})$ with $p \equiv 9, 17 \pmod{24}$ or $p \equiv 1 \pmod{3}$ by Theorem \ref{over quad main thm}, which, in turn, is equivalent to the fact that $p \equiv 17 \pmod{24}$.\\
\indent This completes the proof.
\end{proof}

Combining all the previous results, we can obtain the following
\begin{corollary}\label{simple cor}
Let $X$ be a simple abelian surface over a finite field $k:=\F_q$ with $q=p^a$ for some prime $p$ and an integer $a \geq 1.$ Let $D=\textrm{End}_k^0(X).$ Then the Jordan constant $J_{D^{\times}}$ of $D^{\times}$ is contained in the set $\{1,2,12,24,60 \}.$ 
\end{corollary}
\begin{proof}
By Lemma \ref{poss end alg}, we need to consider the following three cases: \\
(1) If $D$ is a totally definite quaternion algebra over $\Q$ or over a real quadratic field, then the desired assertion follows from Theorems \ref{over Q main thm} and \ref{over quad main thm}. \\
(2) If $D$ is a CM-field of degree $4$, then every finite subgroup of $D^{\times}$ is cyclic, and hence, it follows that $J_{D^{\times}}=1$. \\
(3) If $D$ is a quaternion division algebra over an imaginary quadratic field, then it follows from Theorem \ref{over quad main thm} that $J_{D^{\times}} \in \{1,2,12\}$. \\
\indent This completes the proof.
\end{proof}

The following theorem is the converse of Corollary \ref{simple cor}.
\begin{theorem}\label{realized finite}
Let $n$ be an integer contained in the set $\{1,2,12,24,60\}.$ Then there is a simple abelian surface $X$ over some finite field $k:=\F_q$ with $q=p^a$ for a prime $p$ and an integer $a \geq 1$ such that $J_{D^{\times}}=n$, where $D:=\textrm{End}_k^0(X).$ 
\end{theorem} 
\begin{proof}
We consider each case one by one. More precisely, we will give an irreducible polynomial $h \in \Z[t]$ such that $h^2 = f_X$ for some simple abelian surface $X$ over a finite field $k=\F_q~(q=p^a)$ with the multiplicative subgroup of the endomorphism algebra $\textrm{End}_k^0(X)$ having the desired Jordan constant. We provide a detailed proof for the case of $n=1$, and then we can proceed in a similar fashion with the specified polynomial $h$ for other cases. \\
(1) For $n=1,$ take $h=t^2-73.$ Let $\pi$ be a zero of $h$ so that $\pi$ is a $73$-Weil number. By Theorem \ref{thm HondaTata}, there exists a simple abelian variety $X$ over $k:=\F_{73}$ of dimension $r$ such that $\pi_X$ is conjugate to $\pi$ so that $K:=\Q(\pi_X)=\Q(\sqrt{73}).$ Then since $73$ is totally ramified in $K,$ it follows from both parts of Proposition \ref{local inv} that $D:=\textrm{End}_k^0(X)$ is a quaternion division algebra over $K$, that is ramified exactly at the real places of $K$, and hence, we get $r=2$ by Corollary \ref{cor TateEnd0}-(b). In particular, $X$ is a simple abelian surface over $k$, and $D \cong D_{73,\infty} \otimes_{\Q} K$, which, in turn, implies that $J_{D^{\times}}=1$ by Corollary \ref{quad cor}. \\
(2) For $n=2,$ take $h=t^2-17$. \\
(3) For $n=12,$ take $h=t^2-3$. \\  
(4) For $n=24,$ take $h=t^2-2$. \\
(5) For $n=60,$ take $h=t^2-5$. \\
\indent This completes the proof.
\end{proof}
Another way to view the aforementioned results is the following
\begin{remark}\label{dist rem}
Consider a set $S$ of pairs $(n,p)$ of an integer $n \geq 1$ and a prime $p$ with the property that there is a simple abelian surface $X$ over $k:=\F_{q}$ with $q=p^a$ ($a \geq 1$) and $J_{D^{\times}}=n$, where $D:=\textrm{End}_k^0(X).$ Then we can see that: \\
(i) $n \in \{1,2,12,24,60\}$ (see Corollary \ref{simple cor});\\
(ii) $(1,73), (2,17), (12,3), (24,2), (60,5) \in S$ (see the proof of Theorem \ref{realized finite});\\
(iii) If we choose an element $(n,p)$ randomly from $S,$ then it is most probable that $n =12$ in view of Corollary \ref{quad cor} and the construction in the proof of Theorem \ref{realized finite}.
\end{remark}

\section{Over algebraically closed fields}\label{alg closed case}
In this section, we briefly consider the case when the base field of an abelian surface is algebraically closed of positive characteristic. To this aim, we first recall the following fact from Oort \cite{10}:
\begin{proposition}\label{closed prop}
Let $X$ be a simple abelian surface over an algebraically closed field $k$ of characteristic $p > 0.$ Then the endomorphism algebra $\textrm{End}_k^0(X)$ is of one of the following types: \\
(1) $\Q$; \\
(2) a real quadratic field;\\
(3) an indefinite quaternion algebra over $\Q$;\\
(4) a CM-field of degree $4$.
\end{proposition}
\begin{proof}
For a proof, see \cite[Proposition 6.1]{10}.
\end{proof}

In view of Proposition \ref{closed prop} and Theorems \ref{over Q main thm}, \ref{over quad main thm}, we obtain the following somewhat simple result:  
\begin{theorem}\label{realized closed}
Let $n$ be an integer. Then there exists a simple abelian surface $X$ over some algebraically closed field $k$ of characteristic $p > 0$ such that $J_{D^{\times}}=n,$ where $D:=\textrm{End}_k^0(X)$ if and only if $n=1.$
\end{theorem}
\begin{proof}
Suppose first that there is a simple abelian surface $X$ over an algebraically closed field $k$ of characteristic $p > 0$ with $D:=\textrm{End}_k^0(X).$ Then $D$ is of one of the four types in Proposition \ref{closed prop}. If $D$ is either $\Q$ or a real quadratic field or a quartic CM-field, then every finite subgroup of $D^{\times}$ is cyclic, and hence, we get $J_{D^{\times}} =1.$ If $D$ is an indefinite quaternion algebra over $\Q,$ then since $D$ is not ramified at the infinite place of $\Q$ (by the definition of indefiniteness), it follows from Theorem \ref{over Q main thm} that $J_{D^{\times}}=1.$ Conversely, take $n=1$. Let $p>0$ be a prime and let $D$ be an indefinite quaternion division algebra over $\Q$ such that $D$ is not ramified at $p.$ Then by \cite[Proposition 3.11]{10}, there exists a field $k$ of characteristic $p$ and an absolutely simple abelian surface $X$ over $k$ with $\textrm{End}_{\overline{k}}^0(X_{\overline{k}}) = D,$ where $X_{\overline{k}}:=X \times_k \overline{k}.$ Now, we note that $J_{D^{\times}}=1(=n)$ by Theorem \ref{over Q main thm}. \\ 
\indent This completes the proof. 
\end{proof}

In particular, both of Theorems \ref{realized finite} and \ref{realized closed} assert that every candidate for the Jordan constant is indeed realizable.  

\begin{remark}
In fact, we can obtain a similar result for special higher dimensional cases in the following sense: let $g \geq 3$ be a prime. Let $X$ be a simple abelian variety of dimension $g$ over an algebraically closed field $k$ with $\textrm{char}(k) = p >0.$ Let $D=\textrm{End}_k^0(X).$ Then in view of \cite[$\S$7]{10} or \cite[Remark 2.2]{Hwang(2019)}, $D$ is of one of the following types: \\
(i) $D=\Q$; \\
(ii) $D$ is a totally real field of degree $g$;\\
(iii) $D=D_{p,\infty}$ if (and only if) $g \geq 5$;\\
(iv) $D$ is an imaginary quadratic field;\\
(v) $D$ is a CM-field of degree $2g$;\\
(vi) $D$ is a central simple division algebra of degree $g$ over an imaginary quadratic field and the $p$-rank of $X$ is equal to $0.$ \\
\indent Now, for items (i), (ii), (iv), and (v), we easily see that $J_{D^{\times}}=1.$ For the case when $D=D_{p,\infty}$ and $g \geq 5,$ then $J_{D^{\times}} \in \{1,2,12\}$ by Theorem \ref{over Q main thm}. The only possibly different phenomenon can occur for the item (vi). But then, in light of \cite[Theorem 3.5]{Hwang(2019)}, we know that every finite subgroup of $D^{\times}$ is cyclic, and hence, we also get $J_{D^{\times}}=1.$ Consequently, we can conclude that $J_{D^{\times}}=1$ if $g=3$, and $J_{D^{\times}} \in \{1,2,12\}$ if $g \geq 5.$
\end{remark} 

\section*{Acknowledgement}
The author was supported by a KIAS Individual Grant (MG069901) at Korea Institute for Advanced Study. Also, the author sincerely thanks Professor Prokhorov for making valuable comments on the Introduction of this paper.

\bibliographystyle{amsplain}

\end{document}